\documentclass[12pt,reqno]{amsart}
\usepackage{amsfonts}

\usepackage{amssymb,amsmath,graphicx,amsfonts,euscript}
\usepackage{color}

\newtheorem{thm}{Theorem}[section]

\newtheorem{define}[thm]{Definition}
\newtheorem{rem}[thm]{Remark}

\newtheorem{lemma}[thm]{Lemma}

\allowdisplaybreaks \voffset=-0.2in \numberwithin{equation}{section}

\begin{document}
\bigskip

\centerline{\Large\bf  Global smooth solution to the 2D Boussinesq }
\smallskip

\centerline{\Large\bf  equations with fractional dissipation }

\bigskip

\centerline{Zhuan Ye}

\medskip

\centerline{ School of Mathematical Sciences, Beijing Normal University,}
\medskip

\centerline{Laboratory of Mathematics and Complex Systems, Ministry of Education,}
\medskip

\centerline{Beijing 100875, People's Republic of China}

\medskip

\centerline{E-mail: \texttt{yezhuan815@126.com }}
\medskip

\centerline{\texttt{Tel.: +86 10 58807735; fax: +86 10 58808208.}}

\bigskip
\bigskip

{\bf Abstract:}~~%
In this paper, we consider the two-dimensional (2D) incompressible Boussinesq system
with fractional Laplacian dissipation and thermal diffusion. Based on the previous works and some new observations, we show that the condition $1-\alpha
<\beta<\min\Big\{3-3\alpha,\,\,\frac{\alpha}{2},\,\,
\frac{3\alpha^{2}+4\alpha-4}{8(1-\alpha)}\Big\}$ with
$0.7351\approx\frac{10-2\sqrt{10}}{5}<\alpha<1$ suffices in order for the solution pair of velocity and temperature to remain smooth for all time.

{\vskip 1mm
 {\bf AMS Subject Classification 2010:}\quad 35Q35; 35B65; 76D03.

 }

\vskip .4in
\section{Introduction}\label{intro}
In this paper, we consider the following Cauchy problem of the 2D incompressible Boussinesq equations with the fractional Laplacian dissipation
\begin{equation}\label{Bouss}
\left\{\aligned
&\partial_{t}u+(u \cdot \nabla) u+ \Lambda^{\alpha}u+\nabla p=\theta e_{2},\,\,\,\,\,\,\,\,x\in \mathbb{R}^{2},\,\,t>0, \\
&\partial_{t}\theta+(u \cdot \nabla) \theta+ \Lambda^{\beta}\theta=0, \,\,\,\,\,\,\qquad\qquad x\in \mathbb{R}^{2},\,\,t>0,\\
&\nabla\cdot u=0, \,\,\,\,\,\qquad \qquad \qquad \qquad \qquad \quad x\in \mathbb{R}^{2},\,\,t>0,\\
&u(x, 0)=u_{0}(x),  \quad \theta(x,0)=\theta_{0}(x),\,\,\,\quad x\in \mathbb{R}^{2},
\endaligned\right.
\end{equation}
where $u(x,\,t)=(u_{1}(x,\,t),\,u_{2}(x,\,t))$ is a vector field denoting
the velocity, $\theta=\theta(x,\,t)$ is a scalar function denoting
the temperature in the content of thermal convection and the density
in the modeling of geophysical fluids, $p$ the scalar pressure and
$e_{2}=(0,\,1)$. Here the numbers $\alpha\in [0,\,2]$ and $\beta\in [0,\,2]$ are real parameters. The fractional Laplacian operator $\Lambda^{\alpha}$,
$\Lambda:=(-\Delta)^{\frac{1}{2}}$ denotes the Zygmund operator which is defined through the Fourier transform, namely $$\widehat{\Lambda^{\alpha}
f}(\xi)=|\xi|^{\alpha}\hat{f}(\xi),\qquad \hat{f}(\xi)=\frac{1}{(2\pi)^{2}}\int_{\mathbb{{R}}^{2}}{e^{-ix\cdot\xi}f(x)\,dx}.$$
The fractional Lapacian models many physical phenomena such as overdriven detonations in gases \cite{CLA}. It is also used in some mathematical models in hydrodynamics, molecular biology and finance mathematics,
see for instance \cite{DI}.
We make the convention that by $\alpha=0$ we mean
that there is no dissipation in the velocity equation, and similarly $\beta=0$ means that there is no dissipation in the temperature equation.

The classical 2D Boussinesq equations (i.e.,
$\alpha=\beta=2$) model geophysical flows such as atmospheric fronts
and oceanic circulation, and play an important role in the study of
Rayleigh-Benard convection (see for example \cite{MB,PG} and
references therein). Moreover, there are some geophysical
circumstances related to the Boussinesq equations with fractional
Laplacian (see \cite{Cap,PG} for details). The Boussinesq equations
with fractional Laplacian also closely related equations such as the
surface quasi-geostrophic equation model important geophysical
phenomena (see, e.g., \cite{Constantin}).

\vskip .2in
The standard 2D Boussinesq
equations and their fractional Laplacian generalizations have
attracted considerable attention recently due to their physical
applications and mathematical significance.
On the one hand, when $\alpha=2$ and $\beta=2$, the system
(\ref{Bouss}) becomes the standard 2D Boussinesq equations whose global regularity has been proved previously (see, e.g., \cite{Can}). On the other
hand, the fundamental
issue of whether classical solutions to the totally inviscid Boussinesq equations (i.e., $\alpha=\beta=0$) can develop finite time singularities remains outstandingly open (except
if $\theta_{0}$ is a constant, of course). Therefore, it is very interesting to consider the intermediate cases, that is the fractional dissipation.
We hope that the study of the intermediate cases may shed light on this extremely challenging problem. Almost at the same time, Chae
\cite{C1} and  Hou-Li \cite{HL} successfully established the global regularity to the system (\ref{Bouss}) with $\alpha=2$ and
$\beta=0$ or $\alpha=0$ and $\beta=2$, independently. Xu \cite{XX} proved the global unique solution of the system (\ref{Bouss}) with $\alpha+\beta=2$ and $1\leq\alpha<2$.  For the cases with weaker dissipation, i.e. when $\alpha=0$, $1<\beta<2$ or $1<\alpha<2$, $\beta=0$, the corresponding system (\ref{Bouss}) should be globally well-posed (see \cite{HZ10} and \cite{HK07}, respectively). By deeply
developing the new structures of the coupling system, Hmidi, Keraani and
Rousset \cite{HK3,HK4} were able to establish the global
well-posedness result to the system (\ref{Bouss}) with two special
critical case, namely $\alpha=1$ and $\beta=0$ or $\alpha=0$ and $\beta=1$. The more general critical case $\alpha+\beta=1$ with
$0<\alpha,\,\beta<1$ is extremely difficult. Very recently,  the
global regularity of the general critical case $\alpha+\beta=1$ with
$\alpha>\frac{23-\sqrt{145}}{12}\thickapprox 0.9132$ and $0<\beta<1$
was recently examined by Jiu, Miao, Wu and Zhang \cite{JMWZ}. This
result was further improved by Stefanov and Wu \cite{SW}, which
requires $\alpha+\beta=1$ with
$\alpha>\frac{\sqrt{1777}-23}{24}\thickapprox 0.7981$ and
$0<\beta<1$ (see also \cite{WXXY} for further improvement). Now we want to mention some results concerning the subcritical
ranges, namely $\alpha+\beta>1$ with $0<\alpha,\,\beta<1$.
In fact, the
global regularity of (\ref{Bouss}) in the subcritical
ranges is also definitely nontrivial and
quite difficult. To the best of our knowledge there are
only a few works concerning the subcritical cases, see
\cite{CV,MX,YJW,YeAp,YX201501,YX201502,YXX}. We note that not
all subcritical cases have been resolved.
As a rule of thumb, with current methods it seems impossible to obtain
the global regularity for the 2D Boussinesq equations in the supercritical cases, namely $\alpha+\beta<1$ with $\alpha,\,\beta>0$.
Recently, the eventual regularity of weak solutions of the system (\ref{Bouss}) when
$\alpha$ and $\beta$ are in the suitable supercritical range has been proven (see \cite{JWYang} and \cite{WXXY}). For many other interesting
recent results on the Boussinesq equations, we refer to e.g.
\cite{ACSWXY,ACWX,CW,CW2,ChaeCW,D,DP3,Hmidi2011,JPL,JMWZ,KRTW,LaiPan,LLT,WuXu,WXY,YZ} and in the references therein (the list
with no intention to be complete).

\vskip .2in
To complement and improve the existing results described above,
the goal of this paper is to study the case $1-\alpha<\beta<f(\alpha)$ and see how much
$\alpha>0$ may be reduced while preserving the global regularity result.
The previous three works \cite{MX,YXX,YeAp} required $\alpha>\widetilde{\alpha_{0}}\approx0.7796$ while in this paper we show that $\alpha>\frac{10-2\sqrt{10}}{5}\approx0.7351$ suffices.
More precisely, our main result reads as follows.
\begin{thm}\label{Th1} Let $0.7351\approx\frac{10-2\sqrt{10}}{5}<\alpha<1$ and $1-\alpha<\beta<f(\alpha)$,
where
\begin{eqnarray}\label{NR}f(\alpha)=\min\Big\{3-3\alpha,\,\,\frac{\alpha}{2},\,\,
\frac{3\alpha^{2}+4\alpha-4}{8(1-\alpha)}\Big\}.
\end{eqnarray}
Let $(u_{0}, \theta_{0})
\in H^{\sigma}(\mathbb{R}^{2})\times H^{\sigma}(\mathbb{R}^{2})$ for $\sigma>2$,
then the system (\ref{Bouss}) admits a unique global solution such that for any given
$T>0$
$$u\in C([0, T]; H^{\sigma}(\mathbb{R}^{2}))\cap L^{2}([0, T]; H^{\sigma+\frac{\alpha}{2}}(\mathbb{R}^{2})),$$
$$\theta\in C([0, T]; H^{\sigma}(\mathbb{R}^{2}))\cap L^{2}([0, T];
H^{\sigma+\frac{\beta}{2}}(\mathbb{R}^{2})).$$
\end{thm}
\begin{rem}\rm
Combining the previous three works \cite{MX,YXX,YeAp}, the global well-posedness result of the system (\ref{Bouss}) was established under the
the condition $1-\alpha<\beta<f(\alpha)$ for $\alpha>\widetilde{\alpha_{0}}\thickapprox0.7796$,
where $f(\alpha)<1$ is an explicit function as a technical
bound. Hence, this present result
can be regarded as a further improvement of the results in \cite{MX,YXX,YeAp}.
\end{rem}

\begin{rem}\rm
We want to point out that due to the technical reasons, the $\beta$ is smaller than a complicated explicit function. Indeed, it is strongly
believed that the diffusion term is always good term and the larger the
power $\beta$ is, the better effects it produces. Therefore,
we conjecture that the above theorem should hold
for all the cases $\frac{10-2\sqrt{10}}{5}<\alpha<1$ and $1-\alpha<\beta<1$.
\end{rem}

The rest part of this paper is organized as follows. In Sect.2, we collect some preliminaries materials, including the Littlewood-Paley decomposition, the definition of Besov spaces and
some useful lemmas. In Sect. 3, we give the proof of Theorem \ref{Th1}.
In the Appendix, we give the details about the fact that a choice of $p$ is possible.

\vskip .3in
\section{Preliminaries}
\setcounter{equation}{0}
In this section, we are going to recall some basic facts on the Littlewood-Paley theory, the
definition of Besov space and some useful lemmas. Now we recall the so-called
Littlewood-Paley operators and their elementary properties which
allow us to define the Besov spaces (see for example \cite{BCD,MiaoWZ}).
Let $(\chi, \varphi)$ be a couple of smooth
functions with values in $[0, 1]$ such that $\chi\in
C_{0}^{\infty}(\mathbb{R}^{n})$ is supported in the ball
$\mathcal{B}:= \{\xi\in \mathbb{R}^{n}, |\xi|\leq \frac{4}{3}\}$,
$\rm{\varphi\in C_{0}^{\infty}(\mathbb{R}^{n})}$ is supported in the annulus
$\mathcal{C}:= \{\xi\in \mathbb{R}^{n}, \frac{3}{4}\leq |\xi|\leq
\frac{8}{3}\}$ and satisfy
$$\chi(\xi)+\sum_{j\in \mathbb{N}}\varphi(2^{-j}\xi)=1, \quad  \forall \xi\in \mathbb{R}^{n};\qquad\sum_{j\in \mathbb{Z}}\varphi(2^{-j}\xi)=1,  \  \forall  \xi \in \mathbb{R}^{n}\setminus \{0 \}.$$
For every $ u\in S'$ (tempered distributions)  we define the
non-homogeneous Littlewood-Paley operators as follows,
$$\Delta_{j}u=0\,\,\,\mbox{for}\,\, j\leq -2; \ \quad \Delta_{-1}u=\chi(D)u=\mathcal{F}^{-1}(\chi(\xi)\widehat{u}(\xi)), $$
$$ \Delta_{j}u=\varphi(2^{-j}D)u=\mathcal{F}^{-1}(\varphi(2^{-j}\xi)\widehat{u}(\xi)),\,\,\forall j\in \mathbb{N}.
$$
Meanwhile, we define the homogeneous dyadic blocks as
$$\dot{\Delta}_{j}u=\varphi(2^{-j}D)u=\mathcal{F}^{-1}(\varphi(2^{-j}\xi)\widehat{u}(\xi)),  \ \quad\forall j\in \mathbb{Z}.$$
Let us recall the definition of homogeneous and inhomogeneous Besov spaces through
the dyadic decomposition.
\begin{define}
Let $s\in \mathbb{R}, (p,r)\in[1,+\infty]^{2}$. The homogeneous
Besov space $\dot{B}_{p,r}^{s}$ and inhomogeneous
Besov space $B_{p,r}^{s}$ are defined as a space of $f\in
S'(\mathbb{R}^{n})$ such that
$$ \dot{B}_{p,r}^{s}=\{f\in S'(\mathbb{R}^{n}), \|f\|_{\dot{B}_{p,r}^{s}}<\infty\},\quad B_{p,r}^{s}=\{f\in S'(\mathbb{R}^{n}),  \|f\|_{B_{p,r}^{s}}<\infty\},$$
where
\begin{equation}\label{1}\nonumber
 \|f\|_{\dot{B}_{p,r}^{s}}=\left\{\aligned
&\Big(\sum_{j\in \mathbb{Z}}2^{jrs}\|\dot{\Delta}_{j}f\|_{L^{p}}^{r}\Big)^{\frac{1}{r}}, \quad \forall \ r<\infty,\\
&\sup_{j\in \mathbb{Z}}
2^{js}\|\dot{\Delta}_{j}f\|_{L^{p}}, \quad \forall \ r=\infty,\\
\endaligned\right.
\end{equation}
and
\begin{equation}\label{1}\nonumber
 \|f\|_{B_{p,r}^{s}}=\left\{\aligned
&\Big(\sum_{j\geq-1}2^{jrs}\|\Delta_{j}f\|_{L^{p}}^{r}\Big)^{\frac{1}{r}}, \quad \forall \ r<\infty,\\
&\sup_{j\geq-1}
2^{js}\|\Delta_{j}f\|_{L^{p}}, \quad \forall \ r=\infty.\\
\endaligned\right.
\end{equation}
\end{define}
For $s>0,(p,r)\in[1,+\infty]^{2}$, we have the following fact
$$\|f\|_{ {B}_{p,r}^{s}}= \|f\|_{L^{p}}+\| f\|_{ \dot{B}_{p,r}^{s}}.$$
We shall also need the mixed space-time spaces
$$ \|f\|_{L_{T}^{\rho}B_{p,r}^{s}}:=
\Big\|(2^{js}\|\Delta_{j}f\|_{L^{p}})_{l_{j}^{r}}\Big\|_{L_{T}^{\rho}}$$
and
$$ \|f\|_{\widetilde{L}_{T}^{\rho}B_{p,r}^{s}}:=
(2^{js}\|\Delta_{j}f\|_{L_{T}^{\rho}L^{p}})_{l_{j}^{r}}.$$
The following links are direct consequence of
the Minkowski inequality
$$L_{T}^{\rho}B_{p,r}^{s}\hookrightarrow \widetilde{L}_{T}^{\rho}B_{p,r}^{s},\qquad \mbox{if}\,\,r\geq \rho, \quad \mbox{and} \quad \widetilde{L}_{{T}}^{\rho}B_{p,r}^{s}\hookrightarrow {L}_{T}^{\rho}B_{p,r}^{s},\qquad \mbox{if}\,\,\rho\geq r.$$
In particular,
$$\widetilde{L}_{{T}}^{r}B_{p,r}^{s}\thickapprox {L}_{T}^{r}B_{p,r}^{s}.$$

Bernstein inequalities are fundamental in the analysis involving
Besov spaces and these inequalities trade integrability for
derivatives.

\begin{lemma} [see \cite{BCD}]\label{lem22}
  Let $k\geq0, 1\leq a\leq b\leq\infty$. Assume that
$$
\mbox{supp}\, \widehat{f} \subset \{\xi\in \mathbb{R}^n: \,\, |\xi|
\lesssim  2^j \},
$$
for some integer $j$,  then there exists a constant $C_1$ such that
$$
\|\Lambda^{k} f\|_{L^b} \le C_1\, 2^{j k  +
jn(\frac{1}{a}-\frac{1}{b})} \|f\|_{L^a}.
$$
If $f$ satisfies
\begin{equation*}\label{spp}
\mbox{supp}\, \widehat{f} \subset \{\xi\in \mathbb{R}^n: \,\,|\xi|
\thickapprox 2^j \}
\end{equation*}
for some integer $j$, then
$$
C_1\, 2^{ j k} \|f\|_{L^b } \le \|\Lambda^{k} f\|_{L^b } \le
C_2\, 2^{  j k + j n(\frac{1}{a}-\frac{1}{b})} \|f\|_{L^a},
$$
where $C_1$ and $C_2$ are constants depending on $\alpha,p$ and $q$
only.
\end{lemma}

To prove the theorem, we need the following commutator
estimate involving $\mathcal
{R}_{\alpha}:=\partial_{x}\Lambda^{-\alpha}$, which was established by Stefanov and Wu \cite{SW}.
\begin{lemma}\label{Lem23}
Assume that $\frac{1}{2}<\alpha<1$ and $1<p_{2}<\infty$,
$1<p_{1},\,p_{3}\leq\infty$ with
$\frac{1}{p_{1}}+\frac{1}{p_{2}}+\frac{1}{p_{3}}=1$. Then for $0\leq
s_{1}<1-\alpha$ and $s_{1}+s_{2}>1-\alpha$, the following holds true
\begin{align}\label{SWt201}
\Big|\int_{\mathbb{R}^{2}}{ F[\mathcal
{R}_{\alpha},\,u_{G}\cdot\nabla]\theta\,dx}\Big|\leq
C\|\Lambda^{s_{1}}\theta\|_{L^{p_{1}}}\|F\|_{W^{s_{2},\,p_{2}}}
\|G\|_{L^{p_{3}}}.\end{align} Similarly, for $0\leq s_{1}<1-\alpha$
and $s_{1}+s_{2}>2-2\alpha$, the following holds true
\begin{align}\label{SWt202}
\Big|\int_{\mathbb{R}^{2}}{ F[\mathcal
{R}_{\alpha},\,u_{\theta}\cdot\nabla]H\,dx}\Big|\leq
C\|\Lambda^{s_{1}}\theta\|_{L^{p_{1}}}\|F\|_{W^{s_{2},\,p_{2}}}
\|H\|_{L^{p_{3}}}.\end{align}
Here and in what follows, $W^{s,\,p}$ denotes the standard Sobolev
space.
\end{lemma}
The following lemma contains bilinear estimates (see for example \cite{YeAp,YXX}).
\begin{lemma}\label{Lem25}
Let $2<m<\infty$, $0<s<1$ and $p, q, r\in(1, \infty)^{3}$ such that
$\frac{1}{p}=\frac{1}{q}+\frac{1}{r}$, then it holds
\begin{eqnarray}\label{YXXt205}\|\Lambda^{s}(|f|^{m-2}f)\|_{L^{p}}\leq C\|f\|_{\dot{B}_{q,\,p}^{s}}\|f\|_{L^{r(m-2)}}^{m-2},
\end{eqnarray}
\begin{eqnarray}\||f|^{m-2}f\|_{W^{s,p}}\leq C\|f\|_{{B}_{q,\,p}^{s}}\|f\|_{L^{r(m-2)}}^{m-2}.
\end{eqnarray}
\end{lemma}

The next lemma is the commutator estimate which will be used frequently.
\begin{lemma}[see \cite{MX}]\label{Lem24}
Assume that $u$ is a smooth divergence-free vector
field of $\mathbb{R}^{2}$ and $\theta$ is a smooth function.
 Let $\alpha\in(0,1)$, $s\in(-1, \alpha)$, $p\in[2, \infty]$ and $r\in[1, \infty]$, then
\begin{align}\label{t203}
\|[\mathcal{R}_{\alpha},u\cdot
\nabla]\theta\|_{B_{p,r}^{s}}\leq C\|\nabla
u\|_{L^{p}}(\|\theta\|_{
{B}_{\infty,r}^{1-\alpha+s}}+\|\theta\|_{L^{2}}).
\end{align}
\begin{align}\label{tt204}
\sup_{k\geq-1}2^{(\alpha-1)k}\|\Delta_{k},u\cdot
\nabla]\theta\|_{L^{p}}\leq C(\|\nabla u\|_{B_{p,\infty}^{\alpha-1}}+\|u\|_{L^{2}})\|\theta\|_{B_{\infty,\infty}^{0}}.
\end{align}
\end{lemma}

Finally, we end the section by the following lemma
gathering the smoothing effect of the transport-diffusion equation.
\begin{lemma}[see, e.g., \cite{HK4,MX,YXX}]\label{Lem26}
Consider the following transport-diffusion equation with $0<\beta\leq1$
\begin{equation*}
\left\{\aligned
&\partial_{t}\theta+(u \cdot \nabla) \theta+\Lambda^{\beta}\theta=0,\\
&\nabla\cdot u=0, \\
&\theta(x,0)=\theta_{0}(x),
\endaligned\right.
\end{equation*}
then for any $(p,\,\rho)\in (1,\,\infty)\times [1,\,\infty]$, the following estimate holds
\begin{eqnarray}\label{t206}\sup_{k\in \mathbb{N}}2^{k\frac{\beta}{\rho}}\|\Delta_{k}\theta\|_{L_{t}^{\rho}L^{p}}\leq
C(\|\theta_{0}\|_{L^{p}}+\|\theta_{0}\|_{L^{\infty}}\|\omega\|_{L_{t}^{1}L^{p}}),
\end{eqnarray}
where $\omega$ is the vorticity of the velocity $u$, namely
$\omega=\nabla\times u$.
\end{lemma}

\vskip .4in
\section{The proof of Theorem \ref{Th1}}\setcounter{equation}{0}
This section is devoted to the proof of Theorem \ref{Th1}. First,
the local well posedness of the
system (\ref{Bouss}) for smooth initial data is well-known (see for instance \cite{MB}), and therefore, it suffices to prove the
global in time {\it a priori} estimate on $[0,\,T]$ for any given
$T>0$. Throughout this paper, we denote by $C$ an universal positive
constant whose value may change from line to line. The symbol $C(x,y,z,...)$ means that $C$ depends on variables $x$, $y$, $z$ and so on.

Let us begin with the natural energy estimates of $(u,\,\theta)$. The proof is standard, thus we omit it.
\begin{lemma}\label{NES00}
Assume that $u_{0}\in L^{2}$ and $\theta_{0}\in L^{2}\cap L^{\infty}$. let $(u, \theta)$
be the corresponding solution of the system (\ref{Bouss}).
Then, for any $t>0$, there holds
\begin{eqnarray}\label{t301}
\|\theta(t)\|_{L^{2}}^{2}+2\int_{0}^{t}{
\|\Lambda^{\frac{\beta}{2}}\theta(\tau)\|_{L^{2}}^{2}\,d\tau}\leq
\|\theta_{0}\|_{L^{2}},\quad\|\theta(t)\|_{L^{p}}\leq \|\theta_{0}\|_{L^{p}},\quad \forall p\in
[2, \infty],
\end{eqnarray}
\begin{eqnarray}\label{t302}
\|u(t)\|_{L^{2}}^{2}+2\int_{0}^{t}{
\|\Lambda^{\frac{\alpha}{2}}u(\tau)\|_{L^{2}}^{2}\,d\tau}\leq
(\|u_{0}\|_{L^{2}}+t\|\theta_{0}\|_{L^{2}})^{2}.
\end{eqnarray}
\end{lemma}

In order to obtain the $H^{1}$-bound for $(u,\,\theta)$, we apply operator $\mbox{curl}$ to the equation $(\ref{Bouss})_{1}$
to obtain the following vorticity equation ($w=\partial_{1}u_{2}-\partial_{2}u_{1}$)
\begin{eqnarray}\label{t303}\partial_{t}w+(u\cdot\nabla)w+\Lambda^{\alpha} w=\partial_{x}\theta.\end{eqnarray}
However, the "vortex stretching" term $\partial_{x}\theta$
appears to prevent us from proving any global bound for $w$. To overcome this difficulty, we apply the idea introduced by Hmidi, Keraani and Rousset \cite{HK3,HK4} to eliminate the term $\partial_{x}\theta$ from the vorticity equation.  Now we set $\mathcal {R}_{\alpha}$ as the singular integral operator
$$\mathcal {R}_{\alpha}:=\partial_{x}\Lambda^{-\alpha}.$$
Then we can show that the new quantity $G=\omega-\mathcal {R}_{\alpha}\theta$ satisfies
\begin{eqnarray}\label{t305}\partial_{t}G+(u\cdot\nabla)G+\Lambda^{\alpha}G=[\mathcal {R}_{\alpha},\,u\cdot\nabla]\theta+\Lambda^{\beta-\alpha}\partial_{x}\theta,\end{eqnarray}
where here and in the
sequel the following standard commutator notation is used
$$[\mathcal {R}_{\alpha},\,u\cdot\nabla]\theta:=
\mathcal {R}_{\alpha}(u\cdot\nabla\theta)-u\cdot\nabla\mathcal
{R}_{\alpha}\theta.$$
Moreover, the velocity
field $u$ can be decomposed into the following two parts
$$u=\nabla^{\perp}\Delta^{-1}\omega=\nabla^{\perp}\Delta^{-1}G
+\nabla^{\perp}\Delta^{-1}\mathcal {R}_{\alpha}\theta:=u_{G}+u_{\theta}.$$

The following lemma is concerned with the $L^{2}$ estimate of $G$ and $\Lambda^{\delta}\theta$, which was already established in \cite{YeAp}.
\begin{lemma}\label{Lemma031}
Under the assumptions stated in Theorem \ref{Th1}, let $(u, \theta)$
be the corresponding solution of the system (\ref{Bouss}). If
$\beta>1-\alpha$ and $\alpha>\frac{2}{3}$,
then the temperature $\theta$ admits the following bound
\begin{eqnarray}
\sup_{0\leq t\leq T}(\|G(t)\|_{L^{2}}^{2}+
\|\Lambda^{\delta}\theta(t)\|_{L^{2}}^{2})
+\int_{0}^{T}{\big(\|\Lambda^{\frac{\alpha}{2}} G\|_{L^{2}}^{2}
+\|\Lambda^{\delta+\frac{\beta}{2}}\theta\|_{L^{2}}^{2}\big)(\tau)\,d\tau}\leq
C(T,\,u_{0},\,\theta_{0}),\nonumber
\end{eqnarray}
for any
$\max\Big\{\frac{2-2\alpha-\beta}{2},\,\,\,\frac{2+\beta-3\alpha}{2}\Big\}
<\delta<\frac{\beta}{2}$.
\end{lemma}
\begin{rem}\rm
Although the above Lemma \ref{Lemma031} holds for
$\max\Big\{\frac{2-2\alpha-\beta}{2},\,\,\,\frac{2+\beta-3\alpha}{2}\Big\}
<\delta<\frac{\beta}{2}$, yet by energy estimate (\ref{t301}) and
the classical interpolation, it is actually
true for any $0\leq\delta<\frac{\beta}{2}$. We also remark that $\delta$ can be arbitrarily close to the number $\frac{\beta}{2}$, but at present, we don't know whether Lemma \ref{Lemma031} is true for the case $\delta=\frac{\beta}{2}$.
\end{rem}

With the help of Lemma \ref{Lemma031}, we will show the following global {\it a priori} bound of $L^{m}$ ($2<m<4$) norm for the quantity $G$.
\begin{lemma}\label{Lemma032}
Let $\frac{2}{3}<\alpha<1$ and $1-\alpha<\beta<\frac{\alpha}{2}$.
Assume that $(u_{0},\,\theta_{0})$ satisfies the assumptions stated in Theorem \ref{Th1}, then the combined equation (\ref{t305}) admits the following bound for any $0\leq t\leq T$
\begin{eqnarray}\label{t331}
\|G(t)\|_{L^{m}}^{m}+
\int_{0}^{T}{\|G(\tau)\|_{L^{\frac{2m}{2-\alpha}}}^{m}\,d\tau}\leq
C(T,\,u_{0},\,\theta_{0}),
\end{eqnarray}
where $m$ satisfies the following restriction
\begin{equation}\label{YeCond}
\left\{\aligned
&\frac{4}{3-\alpha-\beta}<m<\min\Big\{4,\,\,\frac{1}{1-\alpha}\Big\}, \\
&\big(2(2-\alpha)\beta-3\alpha+2\big)m<4(2-\alpha)\beta,\\
&(4+8\beta-4\alpha-3\alpha^{2})m<16\beta.
\endaligned\right.
\end{equation}
\end{lemma}
\begin{rem}\rm
It is worthy to emphasize that under the conditions $\frac{2}{3}<\alpha<1$ and $1-\alpha<\beta<\frac{\alpha}{2}$, we indeed can choose some $m\in (2,\,4)$ to guarantee
the above restriction (\ref{YeCond}).
\end{rem}
\begin{proof}[\textbf{Proof of Lemma \ref{Lemma032}}]
Multiplying the equation (\ref{t305}) by
$|G|^{m-2}G$, we have after integrating by part and using the divergence-free condition
\begin{eqnarray}\label{LYe001}&&\frac{1}{m}\frac{d}{dt}\|G(t)\|_{L^{m}}^{m}
+\int_{\mathbb{R}^{2}}
(\Lambda^{\alpha}G)|G|^{m-2}G\,dx\nonumber\\&=&\int_{\mathbb{R}^{2}}
{[\mathcal
{R}_{\alpha},\,u\cdot\nabla]\theta\,\,|G|^{m-2}G\,dx}+\int_{\mathbb{R}^{2}}
{\Lambda^{\beta-\alpha}\partial_{x}\theta\,\,|G|^{m-2}G\,dx}\nonumber\\&=&\int_{\mathbb{R}^{2}}
{[\mathcal
{R}_{\alpha},\,u_{G}\cdot\nabla]\theta\,\,|G|^{m-2}G\,dx}+\int_{\mathbb{R}^{2}}
{[\mathcal
{R}_{\alpha},\,u_{\theta}\cdot\nabla]\theta\,\,|G|^{m-2}G\,dx}\nonumber\\&&
+\int_{\mathbb{R}^{2}}
{\Lambda^{\beta-\alpha}\partial_{x}\theta\,\,|G|^{m-2}G\,dx}\nonumber\\&:=&
N_{1}+N_{2}+N_{3},
\end{eqnarray}
where in the third line we have applied the fact $u=u_{G}+u_{\theta}$.
Thanks to the maximum principle and Sobolev embedding, we deduce that there exists an absolute constant $\widetilde{C}>0$ such that
\begin{eqnarray}\label{LYe002}\int_{\mathbb{R}^{2}}
(\Lambda^{\alpha}G)|G|^{m-2}G\,dx\geq {C}\|\Lambda^{\frac{\alpha}{2}}G^{\frac{m}{2}}\|_{L^{2}}^{2}\geq \widetilde{C}\|G\|_{L^{\frac{2m}{2-\alpha}}}^{m}.\end{eqnarray}
To begin with, we handle the first term $N_{1}$ at the R-H-S of (\ref{LYe001}). As a matter of fact, $N_{1}$ admits a suitable estimate (see \cite{YeAp,YXX}). Now we sketch it here for reader's convenience.
The inequality (\ref{SWt201}) with $s_{1}=0$ as well as the inequality (\ref{YXXt205}) allows us to show
\begin{eqnarray}\label{LYe003} N_{1}&\leq& \Big|\int_{\mathbb{R}^{2}}
{[\mathcal
{R}_{\alpha},\,u_{G}\cdot\nabla]\theta\,\,|G|^{m-2}G\,dx}\Big| \nonumber\\&\leq&
C\|G\|_{L^{q}}\|\theta\|_{L^{\infty}}
\||G|^{m-2}G\|_{W^{s_{2}-\frac{\widetilde{\delta}}{2},\,p}}\nonumber\\
&&(  \mbox{ $\widetilde{\delta}>0$ is small enough})\nonumber\\
&\leq&C\|\theta_{0}\|_{L^{\infty}}
\|G\|_{L^{q}}\|G\|_{L^{(m-2)\times\frac{q}{m-2}}}^{m-2}
\|G\|_{B_{\frac{q}{q-(m-1)},\,p}^{s_{2}-\frac{\widetilde{\delta}}{2}}}
\nonumber\\
&\leq&C\|\theta_{0}\|_{L^{\infty}}
\|G\|_{L^{q}}^{m-1}\|G\|_{B_{\frac{q}{q-(m-1)},\,p}^{s_{2}-\frac{\widetilde{\delta}}{2}}}\nonumber\\
&\leq&C\|\theta_{0}\|_{L^{\infty}}
\|G\|_{L^{q}}^{m-1}\|G\|_{H^{s_{2}-1+\frac{2(m-1)}{q}}},
\end{eqnarray}
where the exponents should satisfy
$$s_{2}-\frac{\widetilde{\delta}}{2}>1-\alpha,\quad \frac{1}{p}+\frac{1}{q}=1,\quad m-1<q< 2(m-1).$$ Moreover, the embedding $H^{s_{2}-1+\frac{2(m-1)}{q}}\hookrightarrow B_{\frac{q}{q-(m-1)},\,p}^{s_{2}-\frac{\widetilde{\delta}}{2}}$ has been used.
Noticing the requirement $s_{2}-\frac{\widetilde{\delta}}{2}>1-\alpha$, one may select a sufficiently small $\widetilde{\delta}>0$ (in fact we can take $\widetilde{\delta}\leq\frac{3\alpha-2}{8}$ for example to satisfy all the conditions) such that
$$s_{2}=1-\alpha+\widetilde{\delta}.$$
By means of the following interpolation inequality
$$\|G\|_{H^{-\alpha+\widetilde{\delta}+\frac{2(m-1)}{q}}}\leq C \|G\|_{L^{2}}^{1-\mu}
\|G\|_{H^{\frac{\alpha}{2}}}^{\mu},\qquad\mu=\frac{-2\alpha+2\widetilde{\delta}+\frac{4(m-1)}{q}}{\alpha}$$
for $$\frac{4(m-1)}{3\alpha-2\widetilde{\delta}}< q< \frac{2(m-1)}{\alpha-\widetilde{\delta}}\Rightarrow \mu\in (0,\,1),$$
we infer that
\begin{eqnarray}\label{LYe004}N_{1}&\leq& C\|\theta_{0}\|_{L^{\infty}}\|G\|_{L^{q}}^{m-1}\|G\|_{L^{2}}^{1-\mu}
\|G\|_{H^{\frac{\alpha}{2}}}^{\mu}\nonumber
\\&\leq&C\|G\|_{L^{q}}^{m-1}
\|G\|_{H^{\frac{\alpha}{2}}}^{\mu}\nonumber
\\&\leq&
C\|G\|_{L^{m}}^{(m-1)(1-\varsigma)}
\|G\|_{L^{\frac{2m}{2-\alpha}}}^{(m-1)\varsigma}
\|G\|_{H^{\frac{\alpha}{2}}}^{\mu}\nonumber\\
&\leq& \frac{\widetilde{C}}{4}\|G\|_{L^{\frac{2m}{2-\alpha}}}^{m}+
C\|G\|_{H^{\frac{\alpha}{2}}}^{\frac{m\mu}{m-(m-1)\varsigma}}
\|G\|_{L^{m}}^{\frac{m(m-1)(1-\varsigma)}{m-(m-1)\varsigma}},
\end{eqnarray}
where we have used the following simple interpolation
\begin{eqnarray}\label{LYe005}
\|G\|_{L^{q}}\leq C\|G\|_{L^{m}}^{1-\varsigma}
\|G\|_{L^{\frac{2m}{2-\alpha}}}^{\varsigma},\quad \varsigma=\frac{2-\frac{2m}{q}}{\alpha},
\end{eqnarray}
for
\begin{eqnarray}\label{LYe006}m<q<
\frac{2m}{2-\alpha}\Rightarrow 0< \varsigma<1.\end{eqnarray}
Noting the following facts
$$ \frac{m(m-1)(1-\varsigma)}{m-(m-1)\varsigma}\leq m,\qquad m\leq\frac{2}{2-2\alpha+\widetilde{\delta}}\Rightarrow \frac{m\mu}{m-(m-1)\varsigma}\leq 2,$$
it follows that
\begin{eqnarray}\label{LYe007}
N_{1}
\leq \frac{\widetilde{C}}{4}\|G\|_{L^{\frac{2m}{2-\alpha}}}^{m}+
C(\|G\|_{H^{\frac{\alpha}{2}}}^{2}+1)
(\|G\|_{L^{m}}^{m}+1).
\end{eqnarray}
We point out that the choice of the number $q$ is possible if we select $\widetilde{\delta}<\frac{3\alpha-2}{2}$. Actually,
combining all the requirement on the number $q$, it should be
$$\max\Big\{m-1,\,\,\frac{4(m-1)}{3\alpha-2\widetilde{\delta}},\,\,m\Big\}<q< \min\Big\{2(m-1),\,\,\frac{2(m-1)}{\alpha-\widetilde{\delta}},\,\,\frac{2m}{2-\alpha}\Big\}.$$
The second term $N_{2}$ and the last term $N_{3}$ at the R-H-S of (\ref{LYe001}) will be treated differently compared with the first term. Here the estimates for the terms $N_{2}$ and $N_{3}$ are the main difference compared to the ones in \cite{YeAp,YXX}.
By making use of the estimate (\ref{SWt202}) with $s_{1}=1+\beta-\alpha-\eta$ and $s_{2}=\eta$, one can show that for any $1<p<2$
\begin{eqnarray}\label{LYe008}N_{2}&\leq&\Big|\int_{\mathbb{R}^{2}}
{[\mathcal
{R}_{\alpha},\,u_{\theta}\cdot\nabla]\theta\,\,|G|^{m-2}G\,dx}\Big|\nonumber\\&\leq&
C\|\Lambda^{1+\beta-\alpha-\eta}\theta\|_{L^{\frac{p}{p-1}}}\|\theta\|_{L^{\infty}}
\||G|^{m-2}G\|_{W^{\eta,\,p}}\nonumber\\
&\leq&C\|\Lambda^{1+\beta-\alpha-\eta}\theta\|_{L^{\frac{p}{p-1}}}
\|\theta_{0}\|_{L^{\infty}} \|G\|_{B_{2,\,p}^{\eta}}\|G\|_{L^{\frac{2p(m-2)}{2-p}}}^{m-2}\nonumber\\
&\leq&C\|\Lambda^{1+\beta-\alpha-\eta}\theta\|_{L^{\frac{p}{p-1}}} \|G
\|_{H^{\frac{\alpha}{2}}}\|G\|_{L^{\frac{2p(m-2)}{2-p}}}^{m-2},
\end{eqnarray}
here and in what follows, we select the parameter $\eta$ such that  \begin{eqnarray}\label{LYe009}\beta<\eta<\min\Big\{1+\beta-\alpha,\,\frac{\alpha}{2}\Big\}.
\end{eqnarray}
Under the assumption (\ref{LYe009}), we are now resorting to the inequality (\ref{YXXt205}) to find
\begin{eqnarray}\label{LYe010}N_{3}&\leq&\Big|\int_{\mathbb{R}^{2}}
{\Lambda^{\beta-\alpha}\partial_{x}\theta\,\,|G|^{m-2}G\,dx}\Big|\nonumber\\&\leq&
C\|\Lambda^{1+\beta-\alpha-\eta}\theta\|_{L^{\frac{p}{p-1}}}
\|\Lambda^{\eta}(|G|^{m-2}G)
\|_{L^{p}}\nonumber\\
&\leq& C\|\Lambda^{1+\beta-\alpha-\eta}\theta\|_{L^{\frac{p}{p-1}}}
\|G\|_{B_{2,\,p}^{\eta}}\|G\|_{L^{\frac{2p(m-2)}{2-p}}}^{m-2}\nonumber\\
&\leq&C\|\Lambda^{1+\beta-\alpha-\eta}\theta\|_{L^{\frac{p}{p-1}}} \|G
\|_{H^{\frac{\alpha}{2}}}\|G\|_{L^{\frac{2p(m-2)}{2-p}}}^{m-2}.
\end{eqnarray}
Now let us recall the following fractional version of the Gagliardo-Nirenberg inequality
\begin{eqnarray}\label{LYe011}
\|\Lambda^{1+\beta-\alpha-\eta}\theta\|_{L^{\frac{p}{p-1}}(\mathbb{R}^{2})}
&\leq& C\|\Lambda^{s}\theta\|_{L^{2}(\mathbb{R}^{2})}^{l}
\|\theta\|_{L^{\infty}(\mathbb{R}^{2})}^{1-l}\qquad (\mbox{see  Theorem 1.2 of \cite{HMOW}})\nonumber\\&\leq& C
\|\theta\|_{L^{2}(\mathbb{R}^{2})}^{\frac{(2\delta+
\beta-2s)l}{2\delta+
\beta}}
\|\Lambda^{\delta+\frac{\beta}{2}}\theta\|_{L^{2}(\mathbb{R}^{2})}^{\frac{2sl}{2\delta+
\beta}}
\|\theta\|_{L^{\infty}(\mathbb{R}^{2})}^{1-l},
\end{eqnarray}
where we need the following restrictions
$$\frac{p-1}{p}-\frac{1+\beta-\alpha-\eta}{2}=\frac{1-s}{2}l\quad \mbox{or} \quad l=\frac{2(p-1)-(1+\beta-\alpha-\eta)p}{(1-s)p}\in (0,\,1),$$
$$1+\beta-\alpha-\eta<sl,\qquad 1+\beta-\alpha-\eta<s<\beta.$$
One can easily check that the above inequality (\ref{LYe011}) holds as long as $p$ satisfies
\begin{eqnarray}\label{LYe012}
\frac{2s}{2s+\alpha+\eta-\beta-1}<p<\frac{2}{\alpha+\eta-\beta+s}.
\end{eqnarray}
Let us also recall the following simple interpolation
\begin{eqnarray}\label{LYe013}
\|G\|_{L^{\frac{2p(m-2)}{2-p}}}\leq C\|G\|_{L^{m}}^{1-\lambda}
\|G\|_{L^{\frac{2m}{2-\alpha}}}^{\lambda},\quad \lambda=\frac{2}{\alpha}-\frac{(2-p)m}{(m-2)\alpha p},
\end{eqnarray}
where
\begin{eqnarray}\label{LYe014}\frac{2m}{3m-4}<p<
\frac{2m}{(2-\alpha)(m-2)+m}\Rightarrow \lambda\in (0,\,1).\end{eqnarray}
The estimates (\ref{LYe011}), (\ref{LYe013}) and the H$\rm\ddot{ o}$lder inequality give directly
\begin{eqnarray}\label{LYe015}&&
C\|\Lambda^{1+\beta-\alpha-\eta}\theta\|_{L^{\frac{p}{p-1}}} \|G
\|_{H^{\frac{\alpha}{2}}}\|G\|_{L^{\frac{2p(m-2)}{2-p}}}^{m-2}\nonumber\\
&\leq&C\|\theta\|_{L^{2} }^{\frac{(2\delta+
\beta-2s)l}{2\delta+
\beta}}
\|\Lambda^{\delta+\frac{\beta}{2}}\theta\|_{L^{2} }^{\frac{2sl}{2\delta+
\beta}}
\|\theta\|_{L^{\infty} }^{1-l} \|G
\|_{H^{\frac{\alpha}{2}}} \|G\|_{L^{m}}^{(m-2)(1-\lambda)}
\|G\|_{L^{\frac{2m}{2-\alpha}}}^{(m-2)\lambda}\nonumber\\
&\leq&C
\|\Lambda^{\delta+\frac{\beta}{2}}\theta\|_{L^{2} }^{\frac{2sl}{2\delta+
\beta}}\|G
\|_{H^{\frac{\alpha}{2}}} \|G\|_{L^{m}}^{(m-2)(1-\lambda)}
\|G\|_{L^{\frac{2m}{2-\alpha}}}^{(m-2)\lambda}\nonumber\\
&\leq& \frac{\widetilde{C}}{8}\|G\|_{L^{\frac{2m}{2-\alpha}}}^{m}+
C(\|\Lambda^{\delta+\frac{\beta}{2}}\theta\|_{L^{2} }^{\frac{2sl}{2\delta+
\beta}}\|G\|_{H^{\frac{\alpha}{2}}})^{\frac{m}{m-(m-2)\lambda}} \|G\|_{L^{m}}^{\frac{m(m-2)(1-\lambda)}{m-(m-2)\lambda}}
\nonumber\\
&\leq& \frac{\widetilde{C}}{8}\|G\|_{L^{\frac{2m}{2-\alpha}}}^{m}+
C(1+\|\Lambda^{\delta+\frac{\beta}{2}}\theta\|_{L^{2} }^{2}+\|G\|_{H^{\frac{\alpha}{2}}}^{2}) (1+\|G\|_{L^{m}}^{m}),
\end{eqnarray}
where in the last line we have used the following requirement
\begin{eqnarray}\label{LYe016}\Big(\frac{2sl}{2\delta+
\beta}+1\Big)\frac{m}{m-(m-2)\lambda}\leq 2.
\end{eqnarray}
Notice that $\delta<\frac{\beta}{2}$ and $\delta$ can be arbitrarily close to the number $\frac{\beta}{2}$, we can check the following requirement instead of (\ref{LYe016})
\begin{eqnarray}\label{LYe017}\Big(\frac{s}{
\beta}l+1\Big)\frac{m}{m-(m-2)\lambda}<2.
\end{eqnarray}
Here we mention that because of the presence of parameter $\delta$ in (\ref{LYe016}), the requirement (\ref{LYe017}) is more simpler than (\ref{LYe016}). Moreover, considering (\ref{LYe017}) will not affect our main result.
Inserting $l$ and $\lambda$ into (\ref{LYe017}), we get the following restriction
\begin{eqnarray}\label{LYe018}
p<\frac{4(1-s)\beta+2\alpha s}{(6-\alpha-\frac{8}{m})(1-s)\beta+\alpha s(1+\alpha+\eta-\beta)}.
\end{eqnarray}
Putting all the restrictions (\ref{LYe012}), (\ref{LYe014}) and (\ref{LYe018}) on $p$ gives
\begin{eqnarray}\label{LYe019}\mathcal{\underline{P}}<p<
\mathcal{\overline{P}}
,\end{eqnarray}
where $$\mathcal{\underline{P}}=\max\Big\{\frac{2m}{3m-4},\,\,\frac{2s}
{2s+\alpha+\eta-\beta-1}\Big\},$$
\begin{eqnarray}\label{LYe020}\mathcal{\overline{P}}&=&\min\Big\{2,
\,\,
\frac{2m}{(2-\alpha)(m-2)+m},
\,\,\frac{2}{\alpha+\eta-\beta+s},\nonumber\\&& \qquad\qquad \frac{4(1-s)\beta+2\alpha s}{(6-\alpha-\frac{8}{m})(1-s)\beta+\alpha s(1+\alpha+\eta-\beta)}\Big\}.
\end{eqnarray}
It should be noted that under the condition (\ref{YeCond}), the number $p$ would work (see the Appendix for a detailed explanation).
The estimate (\ref{LYe015}) ensures
\begin{eqnarray}\label{LYe021}N_{2}+N_{3}\leq
\frac{\widetilde{C}}{4}\|G\|_{L^{\frac{2m}{2-\alpha}}}^{m}+
C(1+\|\Lambda^{\delta+\frac{\beta}{2}}\theta\|_{L^{2} }^{2}+\|G\|_{H^{\frac{\alpha}{2}}}^{2}) (1+\|G\|_{L^{m}}^{m}).
\end{eqnarray}
Substituting the estimates (\ref{LYe002}), (\ref{LYe007}) and
(\ref{LYe021}) into (\ref{LYe001}), it leads to
\begin{eqnarray}\label{LYe022}\frac{d}{dt}\|G(t)\|_{L^{m}}^{m}+
\|G\|_{L^{\frac{2m}{2-\alpha}}}^{m}\leq C(1+\|\Lambda^{\delta+\frac{\beta}{2}}\theta\|_{L^{2} }^{2}+\|G\|_{H^{\frac{\alpha}{2}}}^{2}) (1+\|G\|_{L^{m}}^{m}).
\end{eqnarray}
Thanks to the estimates of Lemma \ref{Lemma031}, the combination of the inequality (\ref{LYe022}) with the Gronwall inequality thus leads to
\begin{eqnarray}\label{LYe023}\|G(t)\|_{L^{m}}^{m}+
\int_{0}^{T}{\|G(\tau)\|_{L^{\frac{2m}{2-\alpha}}}^{m}\,d\tau}\leq
C<\infty.\end{eqnarray}
Therefore, we conclude the proof of Lemma \ref{Lemma032}.
\end{proof}

With the bound (\ref{t331}) in hand, we are now in the position to derive the following lemmas (i.e., Lemmas \ref{Lemma033} and \ref{Lemma034}), which play
an important role in proving the main theorem and are also the main
difference compared to the recent papers \cite{YeAp,YXX}.
\begin{lemma}\label{Lemma033}
Under the assumptions stated in Lemma \ref{Lemma032}, the vorticity $\omega$ admits the following key global {\it a priori} bound
\begin{eqnarray}\label{tNew007}
\int_{0}^{T}{\|\omega(\tau)\|_{L^{\widetilde{m}}}\,d\tau}\leq
C(T,\,u_{0},\,\theta_{0}),\quad 2\leq \widetilde{m}\leq \frac{2m}{2-\alpha},
\end{eqnarray}
where $m$ is the same as in Lemma \ref{Lemma032}.
\end{lemma}
\begin{proof}[\textbf{Proof of Lemma \ref{Lemma033}}]
Notice the fact $\beta>1-\alpha$, then there exist some $\rho>1$ such that $\frac{\beta}{\rho}>1-\alpha$. Recalling the definition of $G$ and the bound (\ref{t331}), we have
\begin{eqnarray}\label{t328}
\|\omega\|_{L_{t}^{1}L^{\widetilde{m}}}&\leq& \|G\|_{L_{t}^{1}L^{\widetilde{m}}}+\|\mathcal {R}_{\alpha}\theta\|_{L_{t}^{1}L^{\widetilde{m}}}\nonumber\\&\leq&
 C(t)+\|\mathcal {R}_{\alpha}\theta\|_{L_{t}^{1}B_{\widetilde{m},1}^{0}}\nonumber\\&\leq&
 C(t)+Ct^{1-\frac{1}{\rho}}\|\mathcal {R}_{\alpha}\theta\|_{\widetilde{L}_{t}^{\rho}B_{\widetilde{m},1}^{0}}.
\end{eqnarray}
The Littlewood-Paley technique and the estimate (\ref{t206}) allow us to show
\begin{eqnarray}\label{t329}
\|\mathcal {R}_{\alpha}\theta\|_{\widetilde{L}_{t}^{\rho}B_{\widetilde{m},1}^{0}}&\leq& \|\Delta_{-1}\mathcal {R}_{\alpha}\theta\|_{\widetilde{L}_{t}^{\rho}B_{\widetilde{m},1}^{0}}
+\|(\mathbb{I}-\Delta_{-1})\theta\|_{\widetilde{L}_{t}^{\rho}B_{\widetilde{m},1}^{1-\alpha}}
\nonumber\\&\leq&
C\|\theta\|_{L_{t}^{\rho}L^{\widetilde{m}}}+
\|(\mathbb{I}-\Delta_{-1})\theta\|_{\widetilde{L}_{t}^{\rho}B_{\widetilde{m},\infty}
^{\frac{\beta}{\rho}}}\nonumber\\&\leq&
C(t)+\sup_{k\in \mathbb{N}}2^{k\frac{\beta}{\rho}}
\|\Delta_{k}\theta\|_{L_{t}^{\rho}L^{\widetilde{m}}}\nonumber\\&\leq&
C(t)+C(\|\theta_{0}\|_{L^{\widetilde{m}}}+\|\theta_{0}\|_{L^{\infty}}
\|\omega\|_{L_{t}^{1}L^{\widetilde{m}}}).
\end{eqnarray}
Combining (\ref{t328}) and (\ref{t329}) yields
\begin{eqnarray}\label{t330}
\|\omega\|_{L_{t}^{1}L^{\widetilde{m}}}\leq C(t)+Ct^{1-\frac{1}{\rho}}\|\omega\|_{L_{t}^{1}L^{\widetilde{m}}},
\end{eqnarray}
where constant $C$ is independent of $t$. Denoting $T_{0}:=(2C)^{-\frac{\rho}{\rho-1}}$,
one can conclude that for any $t\leq T_{0}$
\begin{eqnarray}
\|\omega\|_{L_{t}^{1}L^{\widetilde{m}}}\leq 2C(t).\nonumber
\end{eqnarray}
Adopting the same argument, we get that for any $t\leq T_{0}$
\begin{eqnarray}
\|\omega\|_{L_{[T_{0},t+T_{0}]}^{1}L^{\widetilde{m}}}\leq 2C(t+T_{0}).\nonumber
\end{eqnarray}
By the same iteration, it ensures for any $t\leq T$
\begin{eqnarray}
\|\omega\|_{L_{T}^{1}L^{\widetilde{m}}}\leq C(T).\nonumber
\end{eqnarray}
This completes the proof of Lemma \ref{Lemma033}.
\end{proof}
Based on the estimate (\ref{tNew007}), the next lemma is concerned
with the global {\it a priori} bound
$\int_{0}^{T}{\|G(\tau)\|_{B_{\infty,1}^{0}}\,d\tau}$.
\begin{lemma}\label{Lemma034}
Assume $(u_{0},\,\theta_{0})$ satisfies the assumptions stated in Theorem \ref{Th1}.
Under the assumption $1-\alpha
<\beta<\min\Big\{3-3\alpha,\,\,\frac{\alpha}{2},\,\,
\frac{3\alpha^{2}+4\alpha-4}{8(1-\alpha)}\Big\}$ with
$\frac{10-2\sqrt{10}}{5}<\alpha<1$, the quantity $G$ admits the following key global {\it a priori} bound
\begin{eqnarray}\label{tNew008}
\int_{0}^{T}{\|G(\tau)\|_{B_{\infty,1}^{0}}\,d\tau}\leq
C(T,\,u_{0},\,\theta_{0}).
\end{eqnarray}
\end{lemma}
\begin{proof}[\textbf{Proof of Lemma \ref{Lemma034}}]
Apply inhomogeneous blocks $\Delta_{k}$ ($k\in \mathbb{N}$) operator
to the combined equation (\ref{t305}) to obtain
\begin{eqnarray}\label{t333}\partial_{t}\Delta_{k}G+(u\cdot\nabla)\Delta_{k}G
+\Lambda^{\alpha}\Delta_{k}G=\Delta_{k}[\mathcal {R}_{\alpha},\,u\cdot\nabla]
\theta-[\Delta_{k},\,u\cdot\nabla]G+\Delta_{k}\Lambda^{\beta-\alpha}\partial_{x}\theta.\end{eqnarray}
For notational convenience, we denote
$$f_{k}:=\Delta_{k}[\mathcal {R}_{\alpha},\,u\cdot\nabla]\theta-[\Delta_{k},\,u\cdot\nabla]G
+\Delta_{k}\Lambda^{\beta-\alpha}\partial_{x}\theta$$
Multiplying the equation (\ref{t333}) by $|\Delta_{k}G|^{r-2}\Delta_{k}G$ and using the divergence-free condition, we
can conclude that
\begin{eqnarray}\label{t334}\frac{1}{r}\frac{d}{dt}\|\Delta_{k}G\|_{L^{r}}^{r}
+\int_{\mathbb{R}^{2}}
(\Lambda^{\alpha}\Delta_{k}G)|\Delta_{k}G|^{r-2}\Delta_{k}G\,dx=\int_{\mathbb{R}^{2}}
{f_{k}\,\,|\Delta_{k}G|^{r-2}\Delta_{k}G\,dx},\end{eqnarray}
where $2\leq r\leq \frac{2m}{2-\alpha}$ is to be fixed hereafter. Thanks to (\ref{tNew007}), we have
\begin{eqnarray}\label{OKG}\int_{0}^{T}{\|\omega(\tau)\|_{L^{r}}\,d\tau}<\infty,\qquad2\leq r\leq \frac{2m}{2-\alpha}.
\end{eqnarray}
By the following lower bound (see \cite{CMZ})
$$\int_{\mathbb{R}^{2}}
(\Lambda^{\alpha}\Delta_{k}G)|\Delta_{k}G|^{r-2}\Delta_{k}G\,dx\geq c2^{\alpha k}\|\Delta_{k}G\|_{L^{r}}^{r},\quad k\geq0, $$
for an absolute constant $c>0$ independent of $k$, one arrives at
\begin{eqnarray}\frac{1}{r}\frac{d}{dt}\|\Delta_{k}G\|_{L^{r}}^{r}
+c2^{\alpha k}\|\Delta_{k}G\|_{L^{r}}^{r}\leq C \|f_{k}\|_{L^{r}}
\|\Delta_{k}G\|_{L^{r}}^{r-1}.\nonumber\end{eqnarray} Consequently, making use of the
Gronwall inequality to the above inequality leads to
\begin{eqnarray}\label{t335}\|\Delta_{k}G(t)\|_{L^{r}}
\leq C e^{-ct2^{\alpha k}}\|\Delta_{k}G_{0}\|_{L^{r}}+C\int_{0}^{t}{ e^{-c(t-\tau)2^{\alpha k}}\|f_{k}(\tau)\|_{L^{r}}\,d\tau}.\end{eqnarray}
Integrating over time variable and using the convolution Young inequality yield
\begin{eqnarray}\int_{0}^{t}{\|\Delta_{k}G(\tau)\|_{L^{r}}\,d\tau}
&\leq& C 2^{-\alpha
k}\|\Delta_{k}G_{0}\|_{L^{r}}+C2^{-\alpha k}\int_{0}^{t}{
\|f_{k}(\tau)\|_{L^{r}}\,d\tau}\nonumber\\&\leq&
C 2^{-\alpha
k}\|\Delta_{k}G_{0}\|_{L^{r}}+C2^{-\alpha k}(J_{1}+J_{2}+J_{3}),
\nonumber\end{eqnarray}
where
$$J_{1}=\int_{0}^{t}{ \|\Delta_{k}[\mathcal
{R}_{\alpha},\,u\cdot\nabla]\theta(\tau)\|_{L^{r}}\,d\tau},\quad J_{2}=\int_{0}^{t}{
\|[\Delta_{k},\,u\cdot\nabla]G(\tau)\|_{L^{r}}\,d\tau},$$
$$J_{3}=\int_{0}^{t}{
\|\Delta_{k}\Lambda^{\beta-\alpha}\partial_{x}\theta(\tau)\|_{L^{r}}\,d\tau}.$$
According to the estimate (\ref{t203}) with $s=\alpha-1$, we immediately get
\begin{eqnarray}\label{t336}
J_{1}&\leq& 2^{-(\alpha-1)k}\int_{0}^{t}{2^{(\alpha-1)k} \|\Delta_{k}[\mathcal
{R}_{\alpha},\,u\cdot\nabla]\theta(\tau)\|_{L^{r}}\,d\tau}\nonumber\\
&\leq&C2^{-(\alpha-1)k}\int_{0}^{t}{\|\nabla u(\tau)\|_{L^{r}}(\|\theta_{0}\|_{L^{2}}+\|\theta_{0}\|_{L^{\infty}})\,d\tau}\nonumber\\
&\leq&C2^{-(\alpha-1)k}\int_{0}^{t}{\|\omega(\tau)\|_{L^{r}}\,d\tau}\nonumber\\
&\leq&C(t)2^{-(\alpha-1)k},
\end{eqnarray}
where in the last line we have used the estimate (\ref{OKG}).
By means of the commutator estimate (\ref{tt204}), we find
\begin{eqnarray}\label{t337}
J_{2}&\leq& 2^{-(\alpha-1)k}\int_{0}^{t}{2^{(\alpha-1)k} \|[\Delta_{k},\,u\cdot\nabla]G(\tau)\|_{L^{r}}\,d\tau}\nonumber\\
&\leq&C2^{-(\alpha-1)k}\int_{0}^{t}{(\|\nabla u\|_{B_{r,\infty}^{\alpha-1}}+\|u\|_{L^{2}})\|G\|_{B_{\infty,\infty}^{0}}\,d\tau}
\nonumber\\
&\leq&C2^{-(\alpha-1)k}\int_{0}^{t}{(\|\omega\|_{B_{r,\infty}^{\alpha-1}}+\|u\|_{L^{2}})
\|G\|_{B_{r,\infty}^{\frac{2}{r}}}
\,d\tau}\nonumber\\
&\leq&C2^{-(\alpha-1)k}\int_{0}^{t}{(\|G\|_{B_{r,\infty}^{\alpha-1}}+
\|\mathcal{R}_{\alpha}\theta\|_{B_{r,\infty}^{\alpha-1}}+\|u\|_{L^{2}})
\|G\|_{B_{r,1}^{\frac{2}{r}}}
\,d\tau}
\nonumber\\
&\leq&C2^{-(\alpha-1)k}\int_{0}^{t}{(\|G\|_{L^{m}}+
\|\theta\|_{L^{r}}+\|u\|_{L^{2}}+1)
\|G\|_{B_{r,1}^{\frac{2}{r}}}
\,d\tau}\nonumber\\
&\leq&C(t)2^{-(\alpha-1)k}\int_{0}^{t}{
\|G\|_{B_{r,1}^{\frac{2}{r}}}
\,d\tau},
\end{eqnarray}
where we have used the following estimate
$$\|G\|_{B_{r,\infty}^{\alpha-1}}\leq C\|G\|_{L^{m}}<\infty,\qquad \frac{2}{m}\leq 1-\alpha+\frac{2}{r}.$$
Finally, by the estimate (\ref{t206}) and the estimate (\ref{OKG})
\begin{eqnarray}\label{t338}
J_{3}&\leq&\int_{0}^{t}{
\|\Delta_{k}\Lambda^{\beta-\alpha}\partial_{x}\theta(\tau)\|_{L^{r}}\,d\tau}\nonumber\\
&\leq&C2^{-(\alpha-1)k}\int_{0}^{t}{2^{\beta k}
\|\Delta_{k}\theta(\tau)\|_{L^{r}}\,d\tau}\nonumber\\
&\leq&C2^{-(\alpha-1)k}
(\|\theta_{0}\|_{L^{r}}+\|\theta_{0}\|_{L^{\infty}}\|\omega\|_{L_{t}^{1}L^{r}})\nonumber\\
&\leq&C(t)2^{-(\alpha-1)k}.
\end{eqnarray}
Putting all the above mentioned estimates $J_{1},\,J_{2}$ and $J_{3}$ together, one gets
\begin{eqnarray}\label{t339}\int_{0}^{t}{\|\Delta_{k}G(\tau)\|_{L^{r}}\,d\tau}
&\leq&
C 2^{-\alpha
k}\|\Delta_{k}G_{0}\|_{L^{r}}+C(t)2^{-(2\alpha-1)k}\nonumber\\&&+C(t)
2^{-(2\alpha-1)k}\int_{0}^{t}{
\|G\|_{B_{r,1}^{\frac{2}{r}}}
\,d\tau}.
\end{eqnarray}
In view of the definition of the Besov space, we deduce that
\begin{eqnarray}\label{t340}\|G\|_{L_{t}^{1}B_{r,1}^{\frac{2}{r}}}
&\leq& \sum_{k<N_{0}}2^{\frac{2}{r}k}\|\Delta_{k}G\|_{L_{t}^{1}L^{r}}+\sum_{k\geq N_{0}}2^{\frac{2}{r}k}\|\Delta_{k}G\|_{L_{t}^{1}L^{r}}
\nonumber\\&\leq&
C2^{\frac{2}{r}N_{0}}\|G\|_{L_{t}^{1}L^{r}}+C(t)\sum_{k\geq N_{0}}2^{-(\alpha-\frac{2}{r})
k}\|\Delta_{k}G_{0}\|_{L^{r}}+C(t)\sum_{k\geq N_{0}}2^{-(2\alpha-1-\frac{2}{r})k}
\nonumber\\&&+C(t)\sum_{k\geq N_{0} }2^{-(2\alpha-1-\frac{2}{r})k}\|G\|_{L_{t}^{1}B_{r,1}^{\frac{2}{r}}}\nonumber\\&\leq&
C2^{\frac{2}{r}N_{0}}\|G\|_{L_{t}^{1}L^{r}}+C(t)
+C(t)2^{-(2\alpha-1-\frac{2}{r})N_{0}}\|G\|_{L_{t}^{1}B_{r,1}^{\frac{2}{r}}},
\end{eqnarray}
where we have applied the following restriction
\begin{eqnarray}\label{New02} 2\alpha-1-\frac{2}{r}>0.\end{eqnarray}
Choosing $N_{0}$ as
$$\frac{1}{4}\leq C(t)2^{-(2\alpha-1-\frac{2}{r})N_{0}}\leq\frac{1}{2},$$
we conclude
$$\|G\|_{L_{t}^{1}B_{r,1}^{\frac{2}{r}}}\leq C(t)<\infty.$$
By the embedding theorem, we arrive at
$$\|G\|_{L_{t}^{1}B_{\infty,1}^{0}}\leq C \|G\|_{L_{t}^{1}B_{r,1}^{\frac{2}{r}}}\leq C(t)<\infty.$$
Finally, let us check that the numbers $r$ and $m$ can be fixed. Combining all the requirement on the number $r$, we have
$$2\alpha-1-\frac{2}{r}>0\Rightarrow \frac{2}{r}<2\alpha-1,\qquad \frac{2}{m}\leq 1-\alpha+\frac{2}{r}\Rightarrow  \frac{2}{m}+\alpha-1\leq \frac{2}{r},$$
$$2\leq r\leq \frac{2m}{2-\alpha}\Rightarrow \frac{2-\alpha}{m}\leq \frac{2}{r}\leq 1.$$
Therefore, it gives rise to
$$\max\Big\{\frac{2}{m}+\alpha-1,\,\,\frac{2-\alpha}{m}\Big\}\leq\frac{2}{r}< \min\Big\{2\alpha-1,\,\,1\Big\}=2\alpha-1,$$
which would work as long as
\begin{eqnarray}\label{Condition3}
m>\max\Big\{\frac{2}{\alpha},\,\,\frac{2-\alpha}{2\alpha-1}\Big\}=\frac{2}{\alpha},\qquad \Big(\alpha>\sqrt{3}-1\approx 0.7321\Big).
\end{eqnarray}
Recall the condition (\ref{YeCond}), namely
\begin{equation}\label{Condition4}
\left\{\aligned
&\frac{4}{3-\alpha-\beta}<m<\min\Big\{4,\,\,\frac{1}{1-\alpha}\Big\}, \\
&\big(2(2-\alpha)\beta-3\alpha+2\big)m<4(2-\alpha)\beta,\\
&(4+8\beta-4\alpha-3\alpha^{2})m<16\beta.
\endaligned\right.
\end{equation}
Noticing $\alpha>\frac{2}{3}$, we substitute the number $m=\frac{2}{\alpha}$ into (\ref{Condition4}) to get
\begin{eqnarray}\label{NR2}1-\alpha
<\beta<\min\Big\{3-3\alpha,\,\,\frac{\alpha}{2},\,\,\frac{3\alpha-2}{2(2-\alpha)(1-\alpha)},\,\,
\frac{3\alpha^{2}+4\alpha-4}{8(1-\alpha)}\Big\}.
\end{eqnarray}
Thanks to the following simple fact
$$\frac{3\alpha-2}{2(2-\alpha)(1-\alpha)}>\frac{3\alpha^{2}+4\alpha-4}{8(1-\alpha)},\quad \mbox{for any}\quad \alpha>\frac{2}{3},$$
the restriction (\ref{NR2}) reduces to 
\begin{eqnarray}1-\alpha
<\beta<\min\Big\{3-3\alpha,\,\,\frac{\alpha}{2},\,\,
\frac{3\alpha^{2}+4\alpha-4}{8(1-\alpha)}\Big\}.\nonumber
\end{eqnarray}
It is not difficult to check that the assumption for $\beta$ will work
as long as $$1-\alpha<\frac{3\alpha^{2}+4\alpha-4}{8(1-\alpha)}\Rightarrow\alpha>\frac{10-2\sqrt{10}}{5}
\approx0.7351.$$
Here it is worth particularly mentioning that this is the only place where in the proof we need the key assumption
$\alpha>\frac{10-2\sqrt{10}}{5}$.
If inequality (\ref{Condition4}) holds true when $m=\frac{2}{\alpha}$, then one may take $m=\frac{2}{\alpha}+\epsilon$ for some sufficiently small $\epsilon$ ($\epsilon>0$ may depend on $\alpha$ and $\beta$) such that both inequalities (\ref{Condition4}) and (\ref{Condition3}) are fulfilled. Such a choice of $\epsilon>0$ is possible because both inequalities (\ref{Condition4}) and (\ref{Condition3}) are strict.
Consequently, this completes the proof of Lemma \ref{Lemma034}.
\end{proof}

Finally, we would like to establish the following global
{\it a priori} bound  $\int_{0}^{T}{\|\omega(\tau)\|_{B_{\infty,1}^{0}}\,d\tau}$.
\begin{lemma}\label{Lemma035}
Under the assumptions stated in Lemma \ref{Lemma034}, the vorticity $\omega$ admits the following key global {\it a priori} bound
\begin{eqnarray}\label{tNew008}
\int_{0}^{T}{\|\omega(\tau)\|_{B_{\infty,1}^{0}}\,d\tau}\leq
C(T,\,u_{0},\,\theta_{0}).
\end{eqnarray}
\end{lemma}
\begin{proof}[\textbf{Proof of Lemma \ref{Lemma035}}]
Using the Bernstein inequality and choosing $r_{0}\in
(\frac{2}{\alpha+\beta-1},\,\infty)$, it is clear that
\begin{eqnarray}\label{t342}
\|\omega\|_{L_{t}^{1}B_{\infty,1}^{0}}&\leq& \|G\|_{L_{t}^{1}
B_{\infty,1}^{0}}+\|\mathcal
{R}_{\alpha}\theta\|_{L_{t}^{1}B_{\infty,1}^{0}}\nonumber\\&\leq&
C(t)+\|\Delta_{-1}\mathcal
{R}_{\alpha}\theta\|_{L_{t}^{1}L^{\infty}}+ \sum_{k\in
\mathbb{N}}2^{(1-\alpha-\beta+\frac{2}{r_{0}})k}2^{\beta
k}\|\Delta_{k}\theta\|_{L_{t}^{1}L^{r_{0}}}\nonumber\\&\leq&
C(t)+\big(\sum_{k\in
\mathbb{N}}2^{(1-\alpha-\beta+\frac{2}{r_{0}})k}\big) \sup_{k\in
\mathbb{N}}2^{\beta
k}\|\Delta_{k}\theta\|_{L_{t}^{1}L^{r_{0}}}\nonumber\\&\leq&
C(t)+C\sup_{k\in \mathbb{N}}2^{\beta
k}\|\Delta_{k}\theta\|_{L_{t}^{1}L^{r_{0}}},
\end{eqnarray}
where we have used the following fact
$$\|\Delta_{-1}\mathcal {R}_{\alpha}\theta\|_{L_{t}^{1}L^{\infty}}\leq C\|\Delta_{-1}\Lambda^{-\alpha}\theta\|_{L_{t}^{1}L^{\frac{2}{1-\alpha}}}\leq C\|\theta\|_{L_{t}^{1}L^{2}}\leq
C(t).$$ Again, the estimate (\ref{t206}) ensures
\begin{eqnarray}\label{t343}\sup_{k\in \mathbb{N}}2^{\beta k}
\|\Delta_{k}\theta\|_{L_{t}^{1}L^{r_{0}}}\leq
C(\|\theta_{0}\|_{L^{r_{0}}}+\|\theta_{0}\|_{L^{\infty}}
\|\omega\|_{L_{t}^{1}L^{r_{0}}}).
\end{eqnarray}
By the estimate (\ref{t342}), we obtain
\begin{eqnarray}\label{t344}
\|\omega\|_{L_{t}^{1}L^{r_{0}}}&\leq&
\|G\|_{L_{t}^{1}L^{r_{0}}}+\|\mathcal
{R}_{\alpha}\theta\|_{L_{t}^{1}L^{r_{0}}}\nonumber\\&\leq&
\|G\|_{L_{t}^{1}B_{\infty,1}^{0}\cap L^{2}}+\|\mathcal {R}_{\alpha}
\theta\|_{L_{t}^{1}B_{r_{0},1}^{0}}\nonumber\\&\leq&
C(t)+Ct^{1-\frac{1}{\rho}}\|\mathcal {R}_{\alpha}\theta
\|_{\widetilde{L}_{t}^{\rho}B_{r_{0},1}^{0}}\quad (\rho>1).
\end{eqnarray}
An argument similar to that used in the proof of (\ref{t329}) yields
\begin{eqnarray}\label{t345}
\|\mathcal
{R}_{\alpha}\theta\|_{\widetilde{L}_{t}^{\rho}B_{r_{0},1}^{0}}&\leq&
\|\Delta_{-1}\mathcal
{R}_{\alpha}\theta\|_{\widetilde{L}_{t}^{\rho}B_{r_{0},1}^{0}}
+\|(\mathbb{I}-\Delta_{-1})\theta\|_{\widetilde{L}_{t}^{\rho}B_{r_{0},1}
^{1-\alpha}} \nonumber\\&\leq& C\|\theta\|_{L_{t}^{\rho}L^{r_{0}}}+
\|(\mathbb{I}-\Delta_{-1})\theta\|_{\widetilde{L}_{t}^{\rho}B_{r_{0},\infty}
^{\frac{\beta}{\rho}}}\quad
\Big(\frac{\beta}{\rho}>1-\alpha\Big)\nonumber\\&\leq&
C(t)+\sup_{k\in
\mathbb{N}}2^{k\frac{\beta}{\rho}}\|\Delta_{k}\theta\|_{L_{t}^{\rho}L^{r_{0}}}\nonumber\\&\leq&
C(t)+C(\|\theta_{0}\|_{L^{r_{0}}}+\|\theta_{0}\|_{L^{\infty}}
\|\omega\|_{L_{t}^{1}L^{r_{0}}}).
\end{eqnarray}
By the iterative process as used in proving Lemma \ref{Lemma033}, we
thus get
$$\|\omega\|_{L_{t}^{1}L^{r_{0}}}\leq C(t)<\infty,$$
which along with (\ref{t343}) guarantees that
$$\sup_{k\in \mathbb{N}}2^{\beta k}\|\Delta_{k}\theta\|_{L_{t}^{1}L^{r_{0}}}\leq C(t)<\infty.$$
Thus, we conclude the desired bound (\ref{tNew008}). This ends the
proof of Lemma \ref{Lemma035}.
\end{proof}

Bearing in mind the bound (\ref{tNew008}) and the Littlewood-Paley
technique, it is clear that
\begin{eqnarray}
\int_{0}^{T}{\|\nabla u(\tau)\|_{L^{\infty}}\,d\tau}\leq C
\int_{0}^{T}{(\|u(\tau)\|_{L^{2}}+\|\omega(\tau)\|_{B_{\infty,1}^{0}})\,d\tau}<\infty.\nonumber
\end{eqnarray}
The above estimate is sufficient for us to get the desired results
of Theorem \ref{Th1} (see for example \cite{C1,D,YZ}). As a result, we finish the proof of Theorem
\ref{Th1}.

\vskip .4in
\appendix
\section{details about a choice of $p$}
In this appendix, we will give the details that a choice of $p$ is possible.
As a matter of fact, $p$ would work as long as all the following conditions hold
\begin{eqnarray}\label{A001}\frac{2m}{3m-4}<
\frac{2}{\alpha+\eta-\beta+s}\Rightarrow s<\frac{(3+\beta-\alpha-\eta)m-4}{m},\end{eqnarray}
\begin{eqnarray}\label{A002}\frac{2m}{3m-4}<
\frac{4(1-s)\beta+2\alpha s}{(6-\alpha-\frac{8}{m})(1-s)\beta+\alpha s(1+\alpha+\eta-\beta)}\Rightarrow s<\frac{\beta m}{(\alpha+\eta-2)m+4},\end{eqnarray}
\begin{eqnarray}\label{A003}\frac{2s}
{2s+\alpha+\eta-\beta-1}<
\frac{2m}{(2-\alpha)(m-2)+m}\Rightarrow s>\frac{(1+\beta-\alpha-\eta)m}{m-(2-\alpha)(m-2)},\end{eqnarray}
\begin{eqnarray}\label{A004}\frac{2s}
{2s+\alpha+\eta-\beta-1}<
\frac{2}{\alpha+\eta-\beta+s}\Rightarrow s>1+\beta-\alpha-\eta,\end{eqnarray}
\begin{eqnarray}\label{A005}&&\frac{2s}
{2s+\alpha+\eta-\beta-1}<
\frac{4(1-s)\beta+2\alpha s}{(6-\alpha-\frac{8}{m})(1-s)\beta+\alpha s(1+\alpha+\eta-\beta)}\nonumber\\&&\Rightarrow s>\frac{2(1+\beta-\alpha-\eta)\beta}{\alpha(\alpha+\eta-1)+(\frac{8}{m}-2)\beta}.\end{eqnarray}
Noting the above estimates and $s<\beta$, if the following requirement is true, then the number $s$ would work
\begin{eqnarray}\label{A006}\mathcal{\underline{S}}<s<
\mathcal{\overline{S}}
,\end{eqnarray}
where $$\mathcal{\underline{S}}=\max\Big\{\frac{(1+\beta-\alpha-\eta)m}{m-(2-\alpha)(m-2)},\,\,
\frac{2(1+\beta-\alpha-\eta)\beta}{\alpha(\alpha+\eta-1)+(\frac{8}{m}-2)\beta}\Big\},$$
$$\mathcal{\overline{S}}=\min\Big\{\beta,\,\,\frac{(3+\beta-\alpha-\eta)m-4}{m},\,\,\frac{\beta m}{(\alpha+\eta-2)m+4}\Big\}.$$
We now further assume $\eta$ satisfying
$$\eta<3-\alpha-\frac{4}{m},$$
which implies
$$\frac{(3+\beta-\alpha-\eta)m-4}{m}>\beta,\qquad \frac{\beta m}{(\alpha+\eta-2)m+4}>\beta.$$
We thus get
$$\mathcal{\overline{S}}=\beta.$$
Direct computations yield
\begin{eqnarray}\label{A007}
\frac{(1+\beta-\alpha-\eta)m}{m-(2-\alpha)(m-2)}<
\beta\Rightarrow \eta>\frac{(1+\beta-\alpha)m-[2(2-\alpha)-(1-\alpha)m]\beta}{m},
\end{eqnarray}
\begin{eqnarray}\label{A008}
\frac{2(1+\beta-\alpha-\eta)\beta}{\alpha(\alpha+\eta-1)+(\frac{8}{m}-2)\beta}<
\beta\Rightarrow \eta>\frac{2(1+\beta-\alpha)+\alpha(1-\alpha)-(\frac{8}{m}-2)\beta}{2+\alpha}.
\end{eqnarray}
It thus gives
\begin{eqnarray} \beta<\eta<\min\Big\{1+\beta-\alpha,\,\,\frac{\alpha}{2},\,\,3-\alpha-\frac{4}{m}\Big\}
.\nonumber
\end{eqnarray}
It is obvious to find that $\eta$ should satisfy
$$\mathcal{\underline{A}}<\eta<\mathcal{\overline{A}},$$
where $\mathcal{\underline{A}}$ and $\mathcal{\overline{A}}$ are given by
\begin{eqnarray} \mathcal{\underline{A}}&=&\max\Big\{\beta,\,\,\frac{(1+\beta-\alpha)m-[2(2-\alpha)-(1-\alpha)m]\beta}{m}
,\,\,\nonumber\\&& \qquad   \qquad\frac{2(1+\beta-\alpha)+\alpha(1-\alpha)-(\frac{8}{m}-2)\beta}{2+\alpha}
\Big\},\nonumber
\end{eqnarray}
\begin{eqnarray} \mathcal{\overline{A}}=\min\Big\{1+\beta-\alpha,
\,\,\frac{\alpha}{2},\,\,3-\alpha-\frac{4}{m}\Big\}.\nonumber
\end{eqnarray}
Then, it is easy to check that $\eta$ would work provided that
$$\beta<3-\alpha-\frac{4}{m}\Rightarrow m>\frac{4}{3-\alpha-\beta},$$
$$\frac{(1+\beta-\alpha)m-[2(2-\alpha)-(1-\alpha)m]\beta}{m}<1+\beta-\alpha\Rightarrow
m<\frac{2(2-\alpha)}{1-\alpha},$$
$$\frac{2(1+\beta-\alpha)+\alpha(1-\alpha)-(\frac{8}{m}-2)\beta}{2+\alpha}<1+\beta-\alpha
\Rightarrow m<\frac{8}{2-\alpha}.$$
$$\frac{(1+\beta-\alpha)m-[2(2-\alpha)-(1-\alpha)m]\beta}{m}<\frac{\alpha}{2}\Rightarrow
\big(2(2-\alpha)\beta-3\alpha+2\big)m<4(2-\alpha)\beta,$$
$$\frac{2(1+\beta-\alpha)+\alpha(1-\alpha)-(\frac{8}{m}-2)\beta}{2+\alpha}<\frac{\alpha}{2}
\Rightarrow (4+8\beta-4\alpha-3\alpha^{2})m<16\beta,$$
$$\frac{(1+\beta-\alpha)m-[2(2-\alpha)-(1-\alpha)m]\beta}{m}<3-\alpha-\frac{4}{m}\Rightarrow
m>2,$$
$$\frac{2(1+\beta-\alpha)+\alpha(1-\alpha)-(\frac{8}{m}-2)\beta}{2+\alpha}<3-\alpha-\frac{4}{m}
\Rightarrow m>2.$$
Recall the following restriction
$$m\leq\frac{2}{2-2\alpha+\widetilde{\delta}},\quad \widetilde{\delta}>0,$$
and notice the fact that $\widetilde{\delta}>0$ can be arbitrarily small,
then the $m$ only needs to meet the following condition
$$m<\frac{1}{1-\alpha}.$$
Finally, we get that $m$ should satisfy
\begin{equation}\label{A009}
\left\{\aligned
&\frac{4}{3-\alpha-\beta}<m<\min\Big\{4,\,\,\frac{1}{1-\alpha}\Big\}, \\
&\big(2(2-\alpha)\beta-3\alpha+2\big)m<4(2-\alpha)\beta,\\
&(4+8\beta-4\alpha-3\alpha^{2})m<16\beta.
\endaligned\right.
\end{equation}
Therefore, under the above restriction (\ref{A009}) on $m$, a choice of $p$ is possible.

\end{document}